\documentclass[11pt]{amsart}

\usepackage{amssymb,enumerate,upref,url}

\newcommand{\C}{\mathbb{C}}
\newcommand{\E}{E}
\newcommand{\F}{\mathcal F}
\newcommand{\HS}{\mathcal{H}}

\newcommand{\N}{\mathbb{N}}
\newcommand{\R}{\mathbb{R}}
\newcommand{\Z}{\mathbb{Z}}
\DeclareMathOperator{\Ran}{ran}

\DeclareMathOperator{\tr}{tr}

\DeclareMathOperator{\rk}{rk}

\newcommand{\id}{\mathrm{id}}
\newcommand{\norm}[1]{\lVert #1 \rVert}

\newcommand{\ip}[1]{\left< #1 \right>}

\def\sideremark#1{\ifvmode\leavevmode\fi\vadjust{\vbox to0pt{\vss
 \hbox to 0pt{\hskip\hsize\hskip1em
\vbox{\hsize2cm\tiny\raggedright\pretolerance10000
 \noindent #1\hfill}\hss}\vbox to8pt{\vfil}\vss}}}

\newtheorem{thm}{Theorem}[section]
\newtheorem{lemma}[thm]{Lemma}
\newtheorem{prop}[thm]{Proposition}
\newtheorem{cor}[thm]{Corollary}
\theoremstyle{remark}
\newtheorem{rmk}[thm]{Remark}

\theoremstyle{definition}

\numberwithin{equation}{section}

\begin{document}

\title[Projections and idempotents with fixed diagonal]{Projections and idempotents with fixed diagonal and
the homotopy problem \\ for unit tight frames}

\author[Giol, Kovalev, Larson, Nguyen, Tener]{Julien Giol, Leonid V. Kovalev, David Larson, Nga Nguyen and James E. Tener}

\address{Department of
Mathematics, Bucknell University, Lewisburg, PA 17837}
\email{julien.giol@bucknell.edu}

\address{Department of
Mathematics, 215 Carnegie Building, Syracuse University, Syracuse, NY 13244} \email{lvkovale@syr.edu}

\address{Department of
Mathematics, Texas A{\&}M University, College Station, TX 77843}
\email{larson@math.tamu.edu}

\address{Department of
Mathematics, Texas A{\&}M University, College Station, TX 77843}
\email{nnguyen@math.tamu.edu}

\address{Department of
Mathematics, 970 Evans Hall \#3840,
University of California, Berkeley, CA 94720}
\email{jtener@math.berkeley.edu}

\date{19 April 2010}


\thanks{The fifth author was a participant in an NSF funded REU at Texas A\&M University in the
summer of 2008 in which the other authors were mentors. The fourth author was a doctoral student at Texas A\&M
University and teaching assistant for the REU when this work was accomplished. The second and third authors were
partially supported by grants from the NSF}

 \begin{abstract}
We investigate the topological and metric structure of the set of idempotent operators and projections which have
prescribed diagonal entries with respect to a fixed orthonormal basis of a Hilbert space.
As an application, we settle some cases of conjectures of Larson, Dykema, and Strawn on the connectedness of the set
of unit-norm tight frames.
\end{abstract}

\maketitle

\section{Introduction}

A finite unit-norm tight frame (FUNTF) is a finite sequence of unit vectors $(x_1,\dots,x_k)$ in an $n$-dimensional Hilbert space $\HS$ which has the
following reproducing property:
\begin{equation}\label{funtf1}
y=\frac{n}{k}\sum_{j=1}^{k} \ip{y,x_j} x_j \quad \text{ for all } y\in\HS.
\end{equation}
When $k=n$, the above defines an orthonormal basis in $\HS$. The redundancy inherent in the frames with $k>n$ makes them useful
in signal processing, as the original signal may be recovered after a partial loss in transmission. We refer to~\cite{BF,CK,DFKLOW,DS,GKK} for
background on FUNTF and to~\cite{HKLW,HL,SH} for the general theory of frames.

We denote the set of all $k$-vector unit-norm tight frames in an $n$-dimensional Hilbert space by $\F_{k,n}^{\C}$
or $\F_{k,n}^{\R}$ depending on the base field. When $k=n$, the topology of these sets is well understood.
Indeed, $\F_{n,n}^{\C}$ can be identified with the unitary group $U(n)$ and $\F_{n,n}^{\R}$ with the orthogonal group $O(n)$.
In particular, $\F_{n,n}^{\C}$ is pathwise connected while $\F_{n,n}^{\R}$ has two connected components.
Much less is known about the topology of frames with redundancy, i.e., with $k>n$. 
The third author conjectured in~\cite{Lreu} that $\F_{k,n}^{\C}$ is pathwise connected
whenever $k>n\ge 1$, or, equivalently, all $k$-vector unit-norm tight frames are homotopic.
Dykema and Strawn proved in~\cite{DS} that $\F_{k,1}^{\C}$ is pathwise connected for $k\ge 1$ and
$\F_{k,2}^{\R}$ is pathwise connected for $k\ge 4$. They conjectured that
$\F_{k,n}^{\R}$ is pathwise connected whenever $k\ge n+2\ge 4$.  They also showed that over either field,
the number of connected components remains the same when $n$ is replaced with $k-n$. The latter implies that
$\F_{k,k-1}^{\C}$ and $\F_{k,k-2}^{\R}$ are also pathwise connected. The other cases of the conjecture remained open.

The Grammian operator~\cite{HL} of a FUNTF is a scalar multiple of a projection with constant diagonal, see for instance
Corollary~2.6 in~\cite{DS} or Theorem~3.5 in~\cite{CK}.
Furthermore, $\F_{k,n}$ fibers over the set of projections in $B(\C^k)$ or $B(\R^k)$ with all diagonal entries equal to $n/k$.
The fibers are identified with the orthogonal group, which is connected in the complex case and has two connected components
in the real case. Thus, the topological structure of $\F_{k,n}$ is largely determined by the structure of the set of projections
with a fixed constant diagonal. The latter set is the subject of our first result. We denote by $M_n(\C)$ (resp. $M_n(\R)$) the set of all $n\times n$
matrices with complex (resp. real) entries. When the choice of $\C$ or $\R$ is unimportant, we write simply $M_n$.

\begin{thm}\label{halfproj} The set of projections in $M_{2n}(\C)$ with all diagonal entries equal to $1/2$ is pathwise
connected for all $n\ge 1$.
\end{thm}

Theorem~\ref{halfproj} implies that $\F_{2n,n}^{\C}$ is connected for $n\ge 1$.
In the case of real scalars Theorem~\ref{halfproj} remains true if $n\ge 2$, see Remark~\ref{realcase}. Therefore,
$\F_{2n,n}^{\R}$ has at most two connected components when $n\ge 2$, and its quotient under the natural action of the orthogonal group in $\R^n$
is connected.



We denote by $M_n(\C)$ (resp. $M_n(\R)$) the set of all $n\times n$ matrices with complex (resp. real)
entries. When the choice of $\C$ or $\R$ is unimportant, we write simply $M_n$.
Let $D_n\subset M_n$ be the subalgebra of diagonal matrices.
There is a natural linear operator (conditional expectation) $E\colon M_n\to D_n$ which acts by erasing
off-diagonal entries. Theorem~\ref{halfproj} concerns the preimage of $(1/2)\id $ under the restriction of $E$ to
projections. It is natural to ask if preimages of other matrices are connected as well. We do not have a complete
answer, see however Theorem~\ref{amplification}. The image of the set of projections under $E$ was the subject of
recent papers by Kadison~\cite{Kad02, Kad02b}.

Theorem~\ref{pavableproj} provides a partial extension of Theorem~\ref{halfproj} to infinite-dimensional spaces, where
the notion of connectedness is understood in the sense of norm topology on the space of bounded operators.

Our second main result is a non-self-adjoint version of Theorem~\ref{halfproj}, which applies to idempotent matrices
with an arbitrary fixed diagonal.

\begin{thm}\label{idempotents} For every $d$ in $D_n(\C)$, the set of idempotents $q$ in $M_n(\C)$ such that $E(q)=d$ is pathwise connected.
\end{thm}

Naturally, the set of idempotents $q$ such that $E(q)=d$ is empty for some matrices $d$. Diagonal matrices of the form
$E(q)$ are characterized in Theorem~\ref{diagofidemp}. Our proof of Theorem~\ref{idempotents} involves several results
of independent interest. First, we characterize the diagonals of idempotents with given range
 in terms of the commutator of the range projection $[q]$ (Theorem~\ref{prescribedrange}). Specifically, the
characterization involves the relative commutator $\{[q]\}'\cap D_n$.
Along the way we obtain the following rigidity
result (Theorem~\ref{perturbationdistance}): for every $d\in D_n$ there exists $\epsilon>0$ such that the existence of
an idempotent $q$ with $\norm{E(q)-d}<\epsilon$ implies the existence of another idempotent $q_1$ with $E(q_1)=d$. A
perturbation argument is used to connect an arbitrary idempotent to an idempotent $q$ such that $\{[q]\}'\cap
D_n=\C\,\id$ while preserving the diagonal. Finally, we show that idempotents whose range projection has trivial
relative commutant form a path-connected set.

The paper concludes with Section~\ref{infsection}, where some of our results are extended to operators in separable
infinite-dimensional Hilbert spaces. Whether full analogues of Theorem~\ref{halfproj} and~\ref{idempotents} hold in
infinite dimensions remains open.

\section{Preliminaries}

\subsection{Projections as $2\times 2$ matrices}
\label{prelimproj}

The content of this section is well-known folklore. It is essentially contained in~\cite[Theorem 2]{Hal69}. We include
this discussion for the convenience of the reader, since it is the basis for our proof of Theorem \ref{halfproj}.

When a Hilbert space comes as an orthogonal direct sum of two Hilbert spaces, say $\HS=K\oplus L$, projections $p$ of $B(\HS)$ can be identified with those $2\times 2$ matrices
$$p=
\left(
\begin{matrix} a & b\cr b\sp* & d
\end{matrix}
\right)
$$
where $a,b,d$ are operators in $B(K),B(L,K),B(L)$ respectively, such that
\begin{enumerate}[(i)]
\item $0\leq a \leq \id$ and $0\leq d\leq \id\;;$
\item $|b\sp*|=\sqrt{a(\id-a)}$ and $|b|=\sqrt{d(\id-d)}\;;$
\item $ab=b(\id-d).$
\end{enumerate}

Then it is readily seen that  $\mbox{Ker }\;b\sp*=\mbox{Ker }\;bb\sp*=\mbox{Ker }a(\id-a)= \mbox{Ker }\;(\id-a)\oplus\mbox{Ker }\;a$, hence
$$K=\mbox{Ker }\;(\id-a)\oplus\mbox{Ker }\;a\oplus \left(\mbox{Ker }\;b\sp*\right)\sp\perp.$$
Likewise,
$$L=\mbox{Ker }\;d\oplus\mbox{Ker }\;(\id-d)\oplus \left(\mbox{Ker }\;b\right)\sp\perp.$$

According to these two decompositions, we can write
$$a=\id\oplus 0\oplus a' \;,\quad d=0\oplus \id\oplus d' \quad\mbox{and}\quad b= 0\oplus 0\oplus b',$$
where $b'$, for instance, denotes the restriction of $b$ to $\left(\mbox{Ker }\;b\right)\sp\perp$ which is injective and whose range is dense in $\left(\mbox{Ker }\;b\sp*\right)\sp\perp$.

There is a unique polar decomposition $b'=u'|b'|$, where $u'$ is an isometry from $\left(\mbox{Ker }\;b\right)\sp\perp$ onto $\left(\mbox{Ker }\;b\sp*\right)\sp\perp$. Note that (ii) and (iii) above entail $|b'|=\sqrt{d'(\id-d')}$ and $a'b'=b'(\id-d')$.  Hence $|b'|$ commutes with $d'$ and $a'u'|b'|=u'(\id-d')|b'|$. The range of $|b'|$ being dense in $\left(\mbox{Ker }\;b\right)\sp\perp$, it follows that
$$a'u'=u'(\id-d').$$
Thus the positive injective contractions $a'$ and $\id-d'$ are unitarily equivalent and the same statement holds for $\id-a'$ and $d'$.

\subsection{Diagonal conditional expectation and minimal block decomposition}\label{sec22}

Let $\HS$ be a separable Hilbert space and let us fix an orthonormal basis. Let $\{e_i\}_{i\in I}$ denote the corresponding set of rank one projections. An element $x$ of $B(\HS)$, i.e. a bounded linear operator on $\HS$, can be identified with its matrix with respect to this basis. It is then called diagonal if all of its off-diagonal entries are equal to zero, i.e $e_ixe_j=0$ whenever $i\neq j$. The set $D$ made of these diagonal elements is a maximal
abelian self-adjoint algebra in $B(\HS)$ (it is equal to its commutant). It comes with the so-called \emph{diagonal conditional expectation}
$$E:B(\HS)\to D$$
defined as the idempotent map which erases the off-diagonal entries.

Let $x$ in $B(\HS)$ be fixed and denote $\stackrel{x}{\sim}$ the smallest equivalence relation on $I$ such that $i\stackrel{x}{\sim} j$ whenever $e_ixe_j\ne 0$. Summing the projections $e_i$ over each equivalence class, we obtain an orthogonal decomposition of the unit  $\{f_j\}_{j\in J}$ within $D$. We call
$$x=\sum_{j\in J}xf_j$$
the \emph{minimal block decomposition} of $x$. By construction, the projections $f_j$ commute with $x$. More precisely, these are the minimal projections of the commutative von Neumann algebra $\{x\}'\cap D$. Note
$$\left(\{x\}'\cap D\right)'\simeq\prod_{j\in J} f_jB(\HS)f_j,$$
which justifies the terminology.

Our strategy for the proof of Theorem \ref{idempotents} consists in restricting ourselves to idempotents $q$ which share the same diagonal $E(q)=d$ and the property that $\{q\}'\cap D=\C \id$ or, equivalently, $\left(\{x\}'\cap D\right)'=B(\HS)$.

\section{Projections with diagonal 1/2}\label{parameterization}

\subsection{Proof of Theorem \ref{halfproj}}

\begin{proof}
Let $p$ in $M_{2n}(\C)$ be a projection such that  $E(p)=\id/2$ and write $p$ as a $2\times 2$ matrix
\begin{equation}\label{blockformproj}
p=
\left(
\begin{matrix} a & b\cr b\sp* & d\end{matrix}
\right)
\end{equation}
with coefficients $a,b,d$ in $M_n$. We will now use implicitly the preliminary remarks of \ref{prelimproj}.

By assumption on the diagonal, we have $\mbox{Tr}\;a=\mbox{Tr}\;(\id-d)=n/2$. Since the restriction of $a$ to $\left(\mbox{Ker }\;b\sp*\right)\sp\perp$ and that of  $\id-d$ to $\left(\mbox{Ker }\;b\right)\sp\perp$ are unitarily equivalent, they have equal trace and it follows that $\mbox{dim}\;\mbox{Ker}\;(\id-a)=\mbox{dim}\;\mbox{Ker}\;d$. Considering $\id-p$ instead, the same argument shows that $\mbox{dim}\;\mbox{Ker}\;a=\mbox{dim}\;\mbox{Ker}\;(\id-d)$. In particular, we see that the subspaces $\mbox{Ker}\;b$ and $\mbox{Ker}\;b\sp*$ have the same dimension, hence we can extend the unitary $u'$ to a unitary $u$ in $M_n$ such that
$$p=
\left(
\begin{matrix} a & \sqrt{a(\id-a)}u\cr u\sp *\sqrt{a(\id-a)} & u\sp * (\id-a)u
\end{matrix}
\right).
$$

Now if we put $a_t:=(\id-t)a+(t/2)\id$ in $B(K)$, it is easily seen that the formula
$$p_t:=
\left(
\begin{matrix} a_t & \sqrt{a_t(\id-a_t)}u\cr u\sp *\sqrt{a_t(\id-a_t)} & u\sp * (\id-a_t)u
\end{matrix}
\right).
$$
defines a projection-valued path connecting $p_0=p$ and
$$p_1=
\left(
\begin{matrix} \id/2 & u/2\cr u\sp */2 & \id/2
\end{matrix}
\right)\;,
$$
and such that $E(p_t)=1/2$ for all $t$. The main point is the latter assertion, which readily follows from the linearity of $E$ and the identities
$$a_t=(\id-t)a+t(1/2)\id\;,$$

$$ u\sp * (\id-a_t)u=(\id-t)u\sp *(\id-a) u+(t/2)\id.$$

Finally, by connectedness of the unitary group in $M_n(\C)$, every projection $p$ in $M_{2n}(\C)$ with diagonal $\id/2$
can be connected to
\[q= \left(
\begin{matrix} \id/2 & \id/2\cr \id/2 & \id/2
\end{matrix}
\right)\;. \qedhere\]
\end{proof}

\begin{rmk}\label{extendinfinite} Most of the proof of Theorem~\ref{halfproj} carries over to a separable
Hilbert space $\HS$ over real or complex scalars. Indeed, we used the assumption that
the space is finite-dimensional only to prove that the subspaces $\mbox{Ker } b$ and $\mbox{Ker } b^*$
are of the same dimension. (The unitary group in $B(\HS)$ is known to be path-connected
and even contractible~\cite{Ku65}.) Thus we have the following result:
if $\HS$ is decomposed into a direct sum $K\oplus L$,
then the set of all projections of form~\eqref{blockformproj}
with $E(p)=\id/2$ and $\mbox{dim}\;\mbox{Ker}\;b=\mbox{dim}\;\mbox{Ker}\;b^*$ is path-connected.
\end{rmk}

\begin{rmk}\label{realcase} The set of projections with diagonal $\id/2$ in $M_2(\R)$ is not connected, since it
consists of just two elements
\[\begin{pmatrix}
1/2 & \pm 1/2 \\ \pm 1/2 & 1/2
\end{pmatrix}\]
However, this set is path-connected in $M_{2n}(\R)$ for all $n>1$. Indeed, the unitary group splits into two components:
special unitary group and its complement. If the block $b$ in~\eqref{blockformproj}
is not invertible, then in the proof of Theorem~\ref{halfproj} we can choose the unitary $u$ to have determinant $1$ or $-1$.
Thus, the existence of a projection $p$ with $E(p)=\id/2$ and noninvertible $b$ implies the connectedness of the set.
Such a projection can be easily constructed by including the $4\times 4$ block
\[\begin{pmatrix}
1/2 & 1/2 & 0 & 0 \\
1/2 & 1/2 & 0 & 0 \\
0 & 0 & 1/2 & 1/2 \\
0 & 0 & 1/2 & 1/2
\end{pmatrix}.\]
\end{rmk}

\subsection{Explicit parametrization in $4$ dimensions}

We observe that there are three sets of projections in $M_4(\C)$ with diagonal $\id/2$. First, those whose all entries are non-zero can be parametrized by

$$p=
\left(
\begin{matrix}
1/2 & t_1\bar{\xi_1} & t_2\bar{\xi_2} & t_3\bar{\xi_3} \cr
t_1\xi_1 &  1/2 &  \mp i t_3 \xi_1\bar{\xi_2} & \pm i t_2\xi_1\bar{\xi_3} \cr
t_2 \xi_2 & \pm  i t_3\bar{\xi_1}\xi_2 & 1/2 & \mp i t_1\xi_2\bar{\xi_3} \cr
t_3 \xi_3 & \mp i t_2 \bar{\xi_1}\xi_3 & \pm i t_1\bar{\xi_2}\xi_3 & 1/2
\end{matrix}
\right)
$$
with $\sqrt{t_1^2+t_2^2+t_3^2}=1/2$, $t_j>0$, and $\xi_j$ in $\mathbb{T}$. Then come those with exactly four null entries:

$$p=
\left(
\begin{matrix}
1/2 & t_1\bar{\xi_1} & t_2\bar{\xi_2} & 0 \cr
t_1\xi_1 &  1/2 &  0 & t_2 \bar{\xi_3} \cr
t_2 \xi_2 & 0 & 1/2 & - t_1\bar{\xi_1}\xi_2\bar{\xi_3} \cr
0 & t_2 \xi_3 & - t_1\xi_1\bar{\xi_2}\xi_3 & 1/2
\end{matrix}
\right)
$$
with $\sqrt{t_1^2+t_2^2}=1/2$, $t_j>0$, $\xi_j$ in $\mathbb{T}$, and the two other families obtained by permutation of the basis. Finally, here are those which have eight null entries:

$$p=
\left(
\begin{matrix}
1/2 & \bar{\xi_1}/2 & 0 & 0 \cr
\xi_1/2 &  1/2 &  0 & 0 \cr
0 & 0 & 1/2 & \bar{\xi_2}/2 \cr
0 & 0 & \xi_2/2 & 1/2
\end{matrix}
\right)
$$
with $\xi_j$ in $\mathbb{T}$, and the two other families obtained by permutation of the basis.

It follows readily that the set of diagonal $1/2$ projections is pathwise connected in $M_4(\mathbb{C})$, giving us
an explicit, parametric proof of Theorem~\ref{halfproj} in that case.

In the real case, the latter set restricts to three sets of four projections. Also, there are no $4\times 4$ diagonal 1/2 projections whose entries are all non-zero real numbers. And those with four null entries split into twenty-four paths which connect the twelve extreme projections. For instance, for every $\epsilon_1,\epsilon_2,\epsilon_5,\epsilon_6$ in $\{\pm1\}$ such that $\epsilon_1\epsilon_2\epsilon_5\epsilon_6=-1$, the extreme projections
$$p=
\left(
\begin{matrix}
1/2 & \epsilon_1/2 & 0 & 0 \cr
\epsilon_1/2 &  1/2 &  0 & 0 \cr
0 & 0 & 1/2 & \epsilon_6/2 \cr
0 & 0 & \epsilon_6/2 & 1/2
\end{matrix}
\right)
$$
and
$$q=
\left(
\begin{matrix}
1/2 &0 & \epsilon_2/2 & 0 \cr
0 &  1/2 &  0 & \epsilon_5/2\cr
\epsilon_2/2 & 0 & 1/2 & 0 \cr
0 & \epsilon_5/2 & 0 & 1/2
\end{matrix}
\right)
$$
can be connected by the path
$$
\left(
\begin{matrix}
1/2 &\cos\theta\epsilon_1/2 & \sin\theta\epsilon_2/2 & 0 \cr
\cos\theta\epsilon_1/2 &  1/2 &  0 &  \sin\theta\epsilon_5/2\cr
\sin\theta\epsilon_2/2 & 0 & 1/2 & \cos\theta\epsilon_6/2 \cr
0 & \sin\theta\epsilon_5/2 & \cos\theta\epsilon_6/2 & 1/2
\end{matrix}
\right)
$$
with $\theta$ running from $0$ to $\pi/2$.

We let the reader  check that any two extreme projections can be connected by at most three paths of this type. In particular, diagonal 1/2 projections in $M_4(\mathbb{R})$ form a pathwise connected set.

\section{Further connectedness results for projections with fixed diagonal}

\subsection{Amplification of the $2\times 2$ case}

Here is a generalization of Theorem \ref{halfproj}. The proof is basically the same, so we only insist on the points that differ.

\begin{thm}\label{amplification} For every $d$ in $D_{2n}$ of the type $d=\cos^2\theta e + \sin^2\theta e\sp\perp$ with a rank $n$ projection $e$ in $D_{2n}$ and $\theta$ in $[0,\pi/2]$, the set of projections $p$ in $M_{2n}(\C)$ such that $E(p)=d$ is pathwise connected.
\end{thm}

\begin{proof}Up to a permutation, we can assume that $e$ is the projection onto the span of the first $n$ vectors of the canonical basis. Now let $p$ be a projection in $M_{2n}$, written as a $2\times 2$ matrix over $M_n$ like in the previous section. Since  $\mbox{Tr}\;a=\mbox{Tr}\;(\id-d)=n\cos^2\theta$ and  $\mbox{Tr}\;(\id-a)=\mbox{Tr}\;d=n\sin^2\theta$, there exists a unitary $u$ in $M_n$ such that
$$p=
\left(
\begin{matrix} a & \sqrt{a(\id-a)}u\cr u\sp *\sqrt{a(\id-a)} & u\sp * (\id-a)u
\end{matrix}
\right).
$$
Then we put $a_t:=(\id-t)a+t\cos^2\theta1$ and it simply remains to mimick the rest of the proof of Theorem \ref{halfproj}.
\end{proof}




\section{Diagonals of idempotents}

An idempotent is an operator which is equal to its square.
For $d$ in $D_n$ to be the diagonal of an idempotent $q$ in $M_n$, it is necessary that $\mbox{Tr}\;d = \mbox{rank}\;q$ belongs to the set of integers $\{0,1,\ldots,n\}$. Now is this sufficient? The cases $\mbox{Tr}\;d=0$ and $\mbox{Tr}\;d=n$ have to be treated separately. Since $0$ and $\id$ are the only idempotents with rank $0$ and $n$, respectively, it turns out that $0$ and $1$ are the only possible diagonals of idempotents with trace $0$ and $n$, respectively. The remainder of this section is devoted to proving that for every $d$ in $D_n$ with $\mbox{Tr}\;d$ in $\{1,\ldots,n-1\}$ there exists an idempotent $q$ in $M_n$ such that $E(q)=d$.

The case $\mbox{Tr}\;d=1$ is very easy. Let $\{d_1,\ldots,d_n\}$ denote the set of values on the diagonal of $d$ so that $\sum_{j=1}^{n}d_j=1$. Then consider for instance the matrix $q$ in $M_n$ which is defined by its entries $q_{i,j}:=d_i$. It is readily seen that $q$ is idempotent and that $E(q)=d$.

We proceed by induction on $k$.

Assume it has been proven that for all $n\geq k$ and for all $d$ in $D_n$ with $\mbox{Tr}\;d=k-1$ there exists an idempotent $q$ in $M_n$ such that $E(q)=d$. We now take $n\geq k+1$ and $d$ in $D_n$ with $\mbox{Tr}\;d=k$. Let $\{d_1,\ldots,d_n\}$ denote the set of values on the diagonal of $d$.

If $d_{j_0}=1$ for some $j_0$, then $\sum_{j\neq j_0} d_j=k-1$ and the induction hypothesis, together with an obvious splicing argument, help us find $q$.

Since $d\neq \id$,  there exist at least two indices $i,j$ such that $d_i+d_j\neq 2$ (otherwise, we find that $d_j=1$
for all $j$). Without loss of generality, we can assume that $d_1+d_2\neq 2$ and we put $\lambda:=(d_2-1)/(d_1+d_2-2)$.
Since $(d_1+d_2-1)+d_3+\cdots+d_n=k-1$, we can find by assumption an idempotent $r$ in $M_{n-1}$ such that $E(r)$ has
diagonal values $\{d_1+d_2-1, d_3,\ldots,d_n\}$. Now consider the idempotent
$$\widetilde{q}=
\left(
\begin{matrix} 1 & 0\cr 0 & r
\end{matrix}
\right)
=\left(
\begin{matrix} 1 & 0 & 0 \cr 0 & d_1+d_2-1 & 0 \cr 0&0&*
\end{matrix}
\right)$$
and the invertible element
$$\sigma
=\left(
\begin{matrix} \lambda & \lambda -1 & 0 \cr 1 & 1 & 0 \cr 0&0&\id
\end{matrix}
\right)$$
in $M_n$. Then a straightforward computation shows that the idempotent $q:=\sigma \widetilde{q}\sigma\sp{-1}$ has diagonal $d$.

Thus we have proved:

\begin{thm}\label{diagofidemp}Let $d$ be in $D_n$. Then $d$ is the diagonal of an idempotent in $M_n$ if and only if one of the following holds:
\begin{enumerate}[(i)]
\item $d=0$;
\item $d=\id$;
\item $\mbox{Tr}\;d$ belongs to $\{1,\ldots,n-1\}$.
\end{enumerate}
\end{thm}

\section{Idempotents with prescribed range and diagonal}

Throughout this section, we will work with the conditional expectation $E$ from $M_n$ onto $D_n$, the set of diagonal $n\times n$ matrices. Given an idempotent $p$, recall that the set of idempotents which have the same range as $p$ is equal to the affine subset $p+pM_np\sp\perp$, where $p\sp\perp=\id-p$. We will now investigate the intersections of the latter with the preimages $E\sp{-1}(d)$.

First we determine the range of the linear operator $x\longmapsto E(pxp\sp\perp)$.
The following lemma should be compared to Lemma~4.2 in~\cite{DS}, which concerns the rank of the differential of the conditional expectation
on the Grassmannian manifold.

\begin{lemma}\label{keylemma}Let $p$ be a projection in $M_n$ with minimal block decomposition $p=\sum_{j=1}^{s}pf_j$. Then
$$E(pM_np\sp\perp)=\{d\in D_n\;:\; \mbox{Tr} \;df_j=0,\;1\leq j\leq s\}.$$
\end{lemma}

\begin{rmk}\label{hilbertschmidt}If we give $M_n(\C)$ its Hilbert-Schmidt (or Euclidean) structure via the inner product $\mbox{Tr}\; a\sp*b$, then the previous lemma can be restated by saying that $E(pM_np\sp\perp)$ is equal to the orthogonal complement of $\{p\}'\cap D_n$ in $D_n$.
\end{rmk}

\begin{proof}Let $d$  belong to $E(pM_np\sp\perp)$, say $d=E(pxp\sp\perp)$. Then for all $j=1\ldots,s$ we have $\mbox{Tr}\;df_j=\mbox{Tr}\;pxp\sp\perp=\mbox{Tr}\;p\sp\perp f_jpx$ by commutativity of the trace, hence $\mbox{Tr}\;df_j=0$ since $p\sp\perp f_j p=p\sp\perp p f_j=0$. Thus $E(pM_np\sp\perp)$ is contained in $D_n$ and is orthogonal to the span of the $f_j$'s, namely  $\{p\}'\cap D_n$.\\
Now let $d$ in $D_n$ be orthogonal to $E(pM_np\sp\perp)$.  This means that for all $x$ in $M_n$, we have $\mbox{Tr}\; (pxp\sp\perp)\sp*d=\mbox{Tr}\;x\sp*pdp\sp\perp=0$. Hence $pdp\sp\perp=0$, i.e. the range of $p$ is invariant under $d$. By Lagrange interpolation, we can find a polynomial $g$ such that $d\sp*=g(d)$. Thus $pd\sp*p\sp\perp=pg(d)p\sp\perp=0$. It follows that $d$ commutes with $p$, hence $d$ belongs to $\{p\}'\cap D_n$ and the proof is complete.
\end{proof}

Given a projection $p$, we characterize the diagonals which can be realized as the diagonal of an idempotent with the same range as $p$.

\begin{thm}\label{prescribedrange}Let $p$ be a projection in $M_n$ with minimal block decomposition $p=\sum_{j=1}^{s}pf_j$ . For every diagonal matrix $d$ in $D_n$, the following assertions are equivalent:
\begin{enumerate}[(i)]
\item $d$ belongs to $E(p+pM_np\sp\perp)$;
\item $\mbox{Tr}\; df_j=\mbox{rank}\;pf_j$ for $j=1,\ldots,s$.
\end{enumerate}
\end{thm}

\begin{proof}The first assertion says that $d-E(p)$ belongs to $E(pM_np\sp\perp)$. By Lemma \ref{keylemma}, this is equivalent to the fact that $\mbox{Tr}\;(d-E(p))f_j=0$ for $j=1\ldots,s$. And since $\mbox{Tr}\;E(p)f_j=\mbox{Tr}\;pf_j=\mbox{rank}\;pf_j$, we get the equivalence with the second assertion.
\end{proof}

The case of diagonal $\id/2$ being our original motivation, let us restate the previous result in this particular situation.

\begin{cor}\label{prescribedforhalf}Let $p$ be a projection in $M_{2n}$ with minimal block decomposition $p=\sum_{j=1}^{s}pf_j$. Then there is a diagonal $\id/2$ idempotent in $M_{2n}$ with range equal to that of $p$ if and only if $\mbox{rank}\;f_j=2\mbox{rank}\;pf_j$ for $j=1,\ldots,s$.
\end{cor}

Now that we have characterized the diagonals that belong to $E(p+pM_np\sp\perp)$, we will give a uniform lower estimate of the distance between a diagonal and the closed affine subspaces $E(p+pM_np\sp\perp)$ of $M_n$ which do not contain it.

\begin{thm}\label{perturbationdistance}Let $d$ be a diagonal in $D_n$. Let $S$ be the set of all possible sums of diagonal elements of $d$, i.e. the set of all $\mbox{Tr}\;(de)$ when $e$ runs over all diagonal projections. Put $\gamma:=1$ if $S$ is contained in $\Z$ and $\gamma:=\mbox{dist}(S\setminus\Z,\Z)$ otherwise. Then for all projections $p$ in $M_{2n}$, we have the following alternative: either $d$ belongs to $E(p+pM_np\sp\perp)$ or
$$\mbox{dist}(d,E(p+pM_np\sp\perp))\geq \frac{\gamma}{\lfloor n/2\rfloor}.$$
\end{thm}

\begin{proof}Suppose that $d$ does not belong to $E(p+pM_np\sp\perp)$ and let  $p=\sum_{j=1}^{s}pf_j$ be the minimal block decomposition of $p$. By Theorem \ref{prescribedrange}, there is one $j$ such that  $\mbox{Tr}\; df_j\neq\mbox{rank}\;pf_j$.
\end{proof}

Again, we find it worth restating the result above in the special case of diagonal $\id/2$, in a slightly different form.

\begin{cor}\label{distanceforhalf}Let $q$ be an idempotent in $M_{2n}$. If $\|E(q)-1/2\|<\frac{1}{n}$, then there exists an idempotent $\widetilde{q}$ with diagonal $\id/2$ and with range equal to that of $q$.
\end{cor}

\begin{proof}In this case, the constant $\gamma$ is equal to $1/2$. Let $p$ be the range projection of $q$. Since  $\mbox{dist}(d,E(p+pM_np\sp\perp))< \frac{1}{n}$, Theorem \ref{perturbationdistance} implies that $1/2$ is actually the diagonal of an idempotent $\widetilde{q}$ in $E(p+pM_np\sp\perp)$.
\end{proof}

\section{Connectedness of idempotents with fixed diagonal}

This section is devoted to the proof of Theorem \ref{idempotents}. Like in the previous section, we work with $n\times
n$ matrices and the diagonal conditional expectation $E\colon M_n\to D_n$. But this time, we need to assume that matrices are
taken over the complex field (this assumption is used in Lemma~\ref{complexgrassmannian} only).

Given an idempotent $q$, it will prove convenient to denote $[q]$ its range projection which is given, for instance, by the formula $[q]=q(q+q\sp*-\id)\sp{-1}$.

The key idea in our strategy is to reduce to the case of idempotents for which the algebra $\{[q]\}'\cap D_n$ is trivial. We begin with a simple observation.

\begin{rmk}\label{simpleobs}
The commutative finite-dimensional algebra $\{[q]\}'\cap D_n$ is the span of all diagonal projections which leave the range of $q$ invariant, i.e. those diagonal projections $e$ such that $e[q]=[q]e$ or, equivalently, $q\sp\perp e q=0$.
\end{rmk}

We now proceed to the construction that will allow us to implement the reduction claimed above.

\begin{prop}\label{reduction} Let $q$ be a nontrivial idempotent in $M_n$. If $\mbox{dim}\; \{[q]\}'\cap D_n >1$, then there exists an idempotent $r$ in $M_n$ such that
\begin{enumerate}[(i)]
\item $E(r)=E(q)$;
\item $\{[r]\}'\cap D_n \subsetneq \{[q]\}'\cap D_n$;
\item there is a piecewise affine path consisting of at most two steps from $q$ to $r$ within the set of idempotents with diagonal constant equal to $E(q)$.
\end{enumerate}
\end{prop}

\begin{proof}According to Remark \ref{simpleobs},  the algebra $\{[q]\}'\cap D_n$ is spanned by the diagonal projections $e$ such that $eq=qeq$. By assumption, we can find one such $e$ that is non trivial. We will construct an idempotent $r$ such that $\{[r]\}'\cap D_n \subsetneq\{[q]\}'\cap D_n$, the inclusion being proper because we will arrange for $e$ not to be in $\{[r]\}'\cap D_n$.\\
The first step is to connect $q$ to an idempotent $\widetilde{q}$ which commutes with $e$ and has same range and diagonal as $q$. Note that this leaves the algebra $\{[q]\}'\cap D_n=\{[\widetilde{q}]\}'\cap D_n$ unchanged and that one passes from $q$ to $\widetilde{q}$ by a straight line segment. To do this, we set $\widetilde{q}:=q-x$ with $x=eqe\sp\perp+e\sp\perp qe$. Since $e$ is a diagonal projection, it is clear that $E(x)=0$ so that $E(\widetilde{q})=E(q)$. Using the identity $eq=qeq$, we first check that $\widetilde{q}e=e\widetilde{q}=eqe$. Then we verify that $qx=x$ and $xq=0$, so that $\widetilde{q}q=q$ and $q\widetilde{q}=\widetilde{q}$, which is the algebraic condition for the idempotents $q$ and $\widetilde{q}$ to have the same range.  \\
Since $q$ is assumed to be non trivial, so is $\widetilde{q}$, i.e. $\widetilde{q}\neq 0$ and $\widetilde{q}\sp\perp\neq 0$. Also, we took $e$ non trivial, i.e. $e\neq 0$ and $e\sp\perp\neq 0$. Now if $\widetilde{q}e=0$, we have $\widetilde{q}e\sp\perp=\widetilde{q}$ and $\widetilde{q}\sp\perp e=e$. Likewise, if $\widetilde{q}\sp\perp e\sp\perp=0$, we find that $\widetilde{q}e\sp\perp=e\sp\perp$ and $\widetilde{q}\sp\perp e=\widetilde{q}\sp\perp$. As a consequence, up to replacing $e$ by $e\sp\perp$, we can further assume that $\widetilde{q}e\neq 0$ and $\widetilde{q}\sp\perp e\sp\perp\neq 0$, so that $\widetilde{q}\sp\perp e\sp\perp M_n \widetilde{q}e\neq \{0\}$. We pick now an element $y\neq 0$ in the latter. Note that $y=e\sp\perp y e = \widetilde{q}\sp\perp y \widetilde{q}$. \\
For the second step, we will exhibit an idempotent $r$ with same nullspace and diagonal as $\widetilde{q}$, and such that $\{[r]\}'\cap D_n \subsetneq \{[\widetilde{q}]\}'\cap D_n$. To this aim, we consider the parametrized family of idempotents given by $r_t:=\widetilde{q}+ty$. Since $y=e\sp\perp y e$, we have $E(y)=0$ hence $E(r_t)=E(\widetilde{q})=d$ for all $t$.  Since $y= \widetilde{q}\sp\perp y \widetilde{q}$, we see that $\widetilde{q}y=y$ and $y\widetilde{q}=0$, hence $r_t$ is an idempotent with the same nullspace as $\widetilde{q}$ for all $t$. Also, for all $t\neq 0$, we observe that $e$ does not belong to $\{[r]_t\}'\cap D_n$, since $r_t\sp\perp e r_t=-ty\neq 0$. So it only remains to find a value of $t\neq 0$ for which $\{[r_t]\}'\cap D_n \subset \{[q]\}'\cap D_n$ and we will suffice to take the corresponding $r_t$ for the desired $r$. Actually, we will show that all but finitely many values of $t$ will do.\\
Let $f$ be a diagonal projection and consider the map $g:t\longmapsto r_t\sp\perp f r_t$. Since each matrix coefficient is a polynomial of degree not greater than 2, $g$ is either constant equal to zero or vanishes for at most two distinct values of $t$. So if $f$ does not belong to $\{[q]\}'\cap D_n$ or, in other terms, if $g(0)\neq 0$, we see that $f$ belongs to $\{[r_t]\}'\cap D_n$ for two values of $t$ at most. Because there are only finitely many diagonal projections, we deduce that for all but finitely many values of $t$, the projections that lie in $\{[r_t]\}'\cap D_n$ also belong to $ \{[q]\}'\cap D_n$, hence, in view of Remark \ref{simpleobs},  $\{[r_t]\}'\cap D_n \subset \{[q]\}'\cap D_n$.
\end{proof}

\begin{lemma}\label{complexgrassmannian}
Let us fix $0<k<n$ and let $G(k,n)$ denote the set of projections with rank $k$ in $M_n(\C)$.
The set $\Omega:=\{p\in G(k,n)\;:\; \{p\}'\cap D_n=\C \id\}$ is an open, dense, pathwise connected subset of $G(k,n)$.
\end{lemma}

\begin{proof}As is well-known, the Grassmannian $G(k,n)$ is a connected complex manifold of dimension $k(n-k)$. Now let us take $p$ in $G(k,n)$ and observe that $p$ belongs to $\Omega$ if and only if  there is no nontrivial diagonal projection which commutes with $p$. Hence $G(k,n)\setminus\Omega$ is equal to the union, over all nontrivial projections $e$ in $D_n$, of the subsets $\{p\in G(k,n)\;:\; ep=pe\}$. The latter can in turn be decomposed into the disjoint union of the subsets $F_l=\{p\in G(k,n)\;:\; ep=pe,\,\mbox{rank} ep = l\}$, $l$ running from $0$ to $k$.
Each set $F_l$ can be identified with $G(l,\mbox{rank}\;e)\times G(k-l, n-\mbox{rank}\;e)$, which is a complex manifold of dimension not greater than $k(n-k)-1$.\\
This shows in particular that $G(k,n)\setminus\Omega$ is closed and has empty interior. Being a proper analytic subset of the connected complex manifold $G(k,n)$, this set has pathwise connected complement (cf. Proposition 3 of Section 2.2 in \cite{Ch}).
\end{proof}

Thanks to this result and to the Lemma \ref{keylemma} of the previous section, we will now prove that the sets $\Omega\cap E^{-1}(d)$
are either connected or empty.

\begin{thm}\label{connectednessforirreducible}For every diagonal $d$ in $D_n$, the set of idempotents $q$ such that $E(q)=d$ and $\{[q]\}'\cap D_n=\C\, \id$ is pathwise connected whenever it is not empty.
\end{thm}

\begin{proof}Let $q$ and $r$ be two idempotents in the set under consideration, if not empty. By Lemma \ref{complexgrassmannian} with $k=\mbox{rank}\;q$,
we can find a projection-valued path $p_t$ connecting $[q]$ and $[r]$ within $\Omega$, i.e. such that $\{p_t\}'\cap
D_n=\C\, \id$ for all $t$. Then it follows from Lemma \ref{keylemma} that for all $t$, $E(p_tM_np_t\sp\perp)$ is equal
to $D_n\cap \mbox{Ker\;Tr}$, the orthogonal complement of $\C\, \id$ in $D_n$ with respect to the Hilbert-Schmidt inner
product. Hence each operator
$$D_t\colon M_n\to D_n=D_n\cap \mbox{Ker\;Tr}\;\oplus \C\, \id\;,\quad x\longmapsto E(p_txp_t\sp\perp)$$
is such that $D_tD_t\sp*$ realizes an isomorphism from $D_n\cap \mbox{Ker\;Tr}$ onto itself. Thus
$$C_t:=D_t\sp*(D_tD_t\sp*)\sp{-1}:D_n\cap \mbox{Ker\;Tr}\to M_n$$
defines a continuous path of right inverses for $D_t$, seen as operators from  $M_n$ to $D_n\cap \mbox{Ker\;Tr}$.\\
Now consider the path $x_t:=C_t(d-E(p_t))$ in $M_n$, which is continuous and satisfies $E(p_tx_tp_t\sp\perp)=d-E(p_t)$ for all $t$. Setting $q_t:=p_t+p_tx_tp_t\sp\perp$, we obtain an idempotent-valued path within the desired set, from $q_0=[q]$ to $q_1=[r]$. Since $E$ is linear, it only remains to connect the latter to $q$ and $r$ respectively by straight line segments and we are done.
\end{proof}

\begin{proof}[Proof of Theorem \ref{idempotents}]Let $q$ and $r$ be idempotents with diagonal $d$. By Proposition \ref{reduction}, we can connect
them, within a finite number of affine steps in the set of idempotents with diagonal $d$, to two idempotents,
$\widetilde{q}$ and $\widetilde{r}$ respectively, such that, moreover, $\{[q]\}'\cap D_n=\{[r]\}'\cap D_n=\C\, \id$.
The latter pair can now be connected by Theorem~\ref{connectednessforirreducible} and the proof is complete.
\end{proof}

\section{Extensions to infinite dimensions}\label{infsection}

In this section we extend some of the preceding results to operators on a separable Hilbert space $\HS$ over complex or
real scalars. Recall that by Theorem~\ref{halfproj} the set of projections with $1/2$ on the diagonal is path-connected
in $M_{2n}(\C)$ or in $M_{2n}(\R)$ for $n>1$. Theorem~\ref{pavableproj} is a partial extension of this result to
$B(\HS)$ equipped with the operator norm topology. To state it we need the following definition: an operator $x\in
B(\HS)$ with $\norm{x}=1$ is \emph{$2$-pavable} if there exists a diagonal projection $e$ such that $\norm{exe}<1$ and
$\norm{e^{\perp}xe^{\perp}}<1$. Note that for any projection $p$ the operator $2p-\id$ has norm $1$; in fact, it is a
symmetry (i.e., self-adjoint unitary). See~\cite{CEKP07} for recent results on paving of projections.

\begin{thm}\label{pavableproj} The projections $p\in B(\HS)$ such that $E(p)=\id/2$ and
$2p-\id$ is pavable are pathwise connected within the set of all
projections with diagonal $\id/2$.
\end{thm}

\begin{proof} Let $f$ be a diagonal projection with infinite rank and nullity. Given a projection $p$ as in the statement,
we must find a path (within the set of projections with diagonal $\id/2$) from $p$ to the block matrix
\[p_0 =\begin{pmatrix}
\id/2 & \id/2 \\ \id/2 & \id/2
\end{pmatrix}\]
in which the blocks correspond to $\Ran f$ and $\Ran f^{\perp}$.
Let $e$ be a projection that paves $2p-\id$. Note that $e$ has infinite rank and nullity.
Replacing $e$ with $e^{\perp}$
if necessary, we can ensure that both $ef$ and $e^{\perp}f^{\perp}$ have infinite rank.
Let $\{e_i\}_{i\in\N}$ be the standard basis of $\HS$. Let $\sigma\in B(\HS)$ be a zero-diagonal involution that
acts by permuting the basis elements so that (i) $\{e_i,\sigma e_i\}\subset \Ran (ef)$ for infinitely
many values of $i\in\N$ and (ii) $\{e_i,\sigma e_i\}\subset \Ran (e^{\perp}f^{\perp})$ for infinitely
many values of $i\in\N$. Let us write the projection $p_1:=(\id+\sigma)/2$ in the block form
\begin{equation}\label{pavable1}
p_1=\begin{pmatrix}
a_1 & b_1 \\ b_1^* & d_1
\end{pmatrix}
\end{equation}
with respect to the decomposition $\HS=\Ran f\oplus\Ran f^{\perp}$. The block $b_1$ has infinite-dimensional
kernel which contains all vectors $e_i+\sigma e_i$ such that $\{e_i,\sigma e_i\}\subset \Ran f^{\perp}$. Similarly,
the kernel of $b_1^*$ contains all vectors $e_i+\sigma e_i$ such that $\{e_i,\sigma e_i\}\subset \Ran f$.
By Remark~\ref{extendinfinite} the projection $p_1$ can be connected by an appropriate path to $p_0$.

Now let $p_1$ be represented as in~\eqref{pavable1} but with respect to decomposition $\HS=\Ran e\oplus\Ran e^{\perp}$.
Replacing $f$ with $e$ in the preceding paragraph, we again find that $b_1$ and $b_1^*$ have infinite dimensional kernels.
Writing $p$ in block form with the same decomposition $\HS=\Ran e\oplus\Ran e^{\perp}$, we obtain
\[p=\begin{pmatrix}
a & b \\ b^* & d
\end{pmatrix}\]
with $\norm{2a-\id}<1$ and $\norm{2d-\id}<1$. It follows that both $a(\id-a)$ and $d(\id-d)$ are invertible. As was noted in section~\ref{prelimproj},
this implies the invertibility of $b$ (and $b^*$).
By Remark~\ref{extendinfinite} the projection $p$ can be connected by an appropriate path to $p_1$, and we are done.
\end{proof}

Recall that $B(\HS)$ is the dual of $S_1$, the set of trace-class operators on $\HS$.
This duality induces $w^*$-topology on $B(\HS)$. Recall the definition of the minimal block decomposition of an operator from section~\ref{sec22}.

\begin{thm}\label{infdiag} Given a projection $p\in B(\HS)$,
let $\sum_{j\in J}pf_j$ be its minimal block decomposition. The $w^*$-closure of the set
\[\{\E(q)\colon q^2=q,\ \Ran q=\Ran p\}\] consists of all operators $d\in D(\HS)$
such that $\tr(df_j)=\rk(pf_j)$ whenever $f_j$ has finite rank.
\end{thm}

Since the idempotents $q$ in Theorem~\ref{infdiag} are all of the form $p+pxp^{\perp}$,
the conclusion of Theorem~\ref{infdiag} can be deduced from the following lemma.

\begin{lemma}\label{infrangediag} Let $p\in B(\HS)$ be a projection. Define $\mathcal{D}_p\colon B(\HS)\to B(\HS)$
by $\mathcal{D}_p(x)=E(pxp^{\perp})$. The $w^*$-closure of $\Ran \mathcal{D}_p$ is the space
\begin{equation}\label{ranged1}
\{d\in D(\HS) \colon \tr(dc)=0  \ \forall c\in\{p\}'_D\cap S_1\}.
\end{equation}
\end{lemma}

\begin{proof} Let $N$ be the space in~\eqref{ranged1}.
First we prove that $\Ran \mathcal{D}_p\subseteq N$. If $d = \mathcal{D}_p(x)$ for some
$x\in B(\HS)$, then for each $c\in\{p\}'_D\cap S_1$ we have
\[
\tr(dc)=\tr(pxp^{\perp}c) = \tr(p^{\perp}cpx) = \tr (p^{\perp}pcx)=0
\]
which means that $d \in N$.

Next, suppose that $c\in D(\HS)\cap S_1$ annihilates $\Ran \mathcal{D}_p$.
This means that for any $x\in B(\HS)$ we have $\tr(pxp^{\perp}c)=0$. Since
\[
\tr(pxp^{\perp}c)=\tr(p^{\perp}cpx), \quad \forall x\in B(\HS),
\]
it follows that
\begin{equation}\label{infrd01}
p^{\perp}cp=0,
\end{equation}
i.e., $\Ran p$ is invariant under $c$. Using continuous functional
calculus, we can write $c^*=F(c)$, where $F(z)=\bar z$ on the spectrum $\sigma(c)$.
Note that $\sigma(c)$ has empty interior and connected complement. By Mergelyan's theorem
there exists a sequence of polynomials $P_n$ such that $P_n\to F$ uniformly on $\sigma(c)$.
It follows from~\eqref{infrd01} that $p^{\perp} P_n(c)p=0$ for all $n$. Letting $n\to\infty$,
we obtain $p^{\perp}c^*p=0$. Taking adjoints, we find that $pc=pcp$. Since $pcp=cp$ by~\eqref{infrd01},
$c$ and $p$ commute. Thus $c$ annihilates $N$.
\end{proof}

\begin{rmk} In general, $\Ran \mathcal{D}_p$ is not $w^*$-closed.
Indeed, if $p$ or $p^{\perp}$ has finite rank, then $\Ran \mathcal{D}_p$ is contained in $S_1$, although it may be
$w^*$-dense in $D(\HS)$. Therefore, Theorem~\ref{infdiag} does not completely describe the possible diagonals of
idempotents with a given range. The difficulty of obtaining such a description can be illustrated by the following
fact: there exists a nonzero idempotent with zero diagonal~\cite[Theorem 3.7]{La85}.
\end{rmk}

Concerning the connectedness of idempotents sharing the same diagonals, we have the following result
which generalizes Proposition~\ref{reduction}. Recall that $\{[r]\}'_D$ means the range of the diagonal expectation $E$ restricted
to the commutant of the range projection of $r$,
i.e $\{[r]\}'_D=\{[r]\}'\cap D(\HS)$.

\begin{prop}\label{unblock} For any idempotent $q\in B(\HS)\setminus \{0, 1\}$ there exists an idempotent $r\in B(\HS)$ such that
$\{[r]\}'_D=\C\, \id$ and there is a piecewise linear path from $q$ to $r$ within the set of idempotents with diagonal
$\E(q)$.
\end{prop}

The proof is preceded by two lemmas.

\begin{lemma}\label{unblockl1} For any idempotent $q\in B(\HS)$ there exists an idempotent $\tilde q\in (\{[q]\}'_D)'$ such that
$\Ran \tilde q=\Ran q$ and $\E(\tilde q)=\E(q)$.
\end{lemma}
\begin{proof} Let $\{f_j\}_{j\in J}$ be the set of minimal projections in $\{[q]\}'_D$.
Let $\tilde q=\sum_{j\in J} qf_j$ be the expectation of $q$ with respect to the block-diagonal
algebra $(\{[q]\}'_D)'$.
It is easy to see that $\tilde q$ is an idempotent and $\E(\tilde q)=\E(q)$. Since $q\tilde q=\tilde q$ and $\tilde q q=q$,
we have $\Ran \tilde q=\Ran q$.  Finally, $\tilde qf_j=f_jqf_j=f_j\tilde q$ for all $j\in J$, which means $\tilde q\in (\{[q]\}'_D)'$.
\end{proof}

\begin{lemma}\label{unblockl2} Suppose that $t\mapsto x(t)$ is a real analytic map from $\R$ to $B(\HS)$.
Then there exists a countable set $C\subset \R$ such that all operators $x(t)$, $t\in\R\setminus C$, have
the same minimal block decomposition.
\end{lemma}
\begin{proof} Each entry of the matrix representing $x(t)$ in the canonical basis of $\HS$ is a real-analytic
function of $t$. Recall that a scalar-valued real-analytic function has at most countably many zeroes
unless it vanishes identically. Therefore, the set of nonzero entries in the matrix of $x(t)$
is the same for all but countably many values of $t$. Since the set of nonzero entries
determines the minimal block decomposition, the claim follows.
\end{proof}

\begin{proof}[Proof of Proposition~\ref{unblock}] Let $\{f_j\}_{j\in J}$ be the set of minimal projections in $\{[q]\}'_D$.
By virtue of Lemma~\ref{unblockl1} we may assume that
\begin{equation}\label{unblock8}
q\in (\{[q]\}'_D)', \quad \text{i.e., }q=\sum_{j\in J}f_jqf_j.
\end{equation}
For $k,l \in J$ we set $y_{k l}=f_k q^{\perp}x_{k l} qf_l$, where
$x_{k l}\in B(\HS)$ is chosen as follows.
If either $k=l$, $q^{\perp}f_k=0$, or $qf_l=0$, then set $x_{k l}=0$. Otherwise, choose $x_{k l}$
so that $0<\|y_{k l}\|< 2^{-k-l}$. Let $n=\sum_{k,l\in J} y_{k l}$.

One can easily check that $qn=0$, $nq=n$, and $\E(n)=0$. Therefore, $q_t:=q+tn$ is an idempotent
for all $t\in\R$, and $\E(q_t)=\E(q)$. The range projection of $q_t$,
\[[q_t]=(q+tn)\left(q+q^*+t(n+n^*)-\id\right)^{-1}\]
is real analytic in $t$. By Lemma~\ref{unblockl2} $[q_t]$ has the same minimal block
decomposition for all $t\in\R\setminus C$ where $C$ is countable. If this decomposition consists of just one
block, then we can set $r=q+tn$ for some $t\in\R\setminus C$.

Suppose that the minimal block decomposition of $[q_t]$, $t\in\R\setminus C$, is nontrivial.
Then there exists a diagonal projection $f\notin\{0,\id\}$ that commutes with $[q_t]$ for all $t\in \R\setminus C$,
hence for all $t\in\R$. This can be expressed as
\begin{equation}\label{unblock10}
(q^{\perp}-tn)f(q+tn)=0,\quad t\in\R.
\end{equation}
The coefficient of $t$ in~\eqref{unblock10} must be zero, hence
\[
q^{\perp}fn-nfq=0.
\]
Since $f\in {[q]}'_D$, we have $f=\sum_{j\in K}f_j$ for some $K\subset J$.
Also, $f$ commutes with $q$ due to~\eqref{unblock8}. Thus we obtain
\[
0=q^{\perp}fn-nfq=fq^{\perp}n-nqf=fn-nf=fnf^{\perp}-f^{\perp}nf,
\]
hence $fnf^{\perp}=f^{\perp}nf=0$. From the definition of $n$ one can see that $fnf^{\perp}=0$
only if
\begin{equation}\label{unblock12}
q^{\perp}f=0\text{ \ or \ }qf^{\perp}=0.
\end{equation}
Similarly, $f^{\perp}nf=0$ implies
\begin{equation}\label{unblock13}
q^{\perp}f^{\perp}=0\text{ \ or \ }qf=0.
\end{equation}
Since $q,f\notin\{0,\id\}$, the relations~\eqref{unblock12}--\eqref{unblock13} contradict each other.
This completes the proof.
\end{proof}

\end{document}